\pdfoutput=1
\documentclass[11pt,reqno]{amsart}

\usepackage[utf8]{inputenc}
\usepackage[T1]{fontenc}

\usepackage[british]{babel}
\usepackage{mathtools,amssymb,amsthm}
\usepackage[shortlabels]{enumitem}
\usepackage{booktabs}
\usepackage{bookmark}
\usepackage[justification=centering]{caption}
\usepackage[usenames,dvipsnames]{xcolor}
\usepackage{cleveref}

\usepackage{microtype}

\hypersetup{
	colorlinks,
	linkcolor={red!60!black},
	citecolor={green!60!black},
	urlcolor={blue!60!black},
}

\newcommand*{\eps}{\varepsilon}
\renewcommand*{\ge}{\geqslant}
\renewcommand*{\le}{\leqslant}
\newcommand*{\bigO}{\mathcal{O}}
\newcommand*{\abs}[1]{\mathopen{}\left\vert #1 \right\vert\mathclose{}}
\newcommand*{\ssize}[1]{\mathopen{}\left\vert #1 \right\vert\mathclose{}}

\newtheorem*{theorem*}{Theorem}
\newtheorem{theorem}{Theorem}

\newtheorem{corollary}[theorem]{Corollary}
\newtheorem{defn}[theorem]{Definition}
\newtheorem*{defn*}{Definition}

\newtheorem*{prop*}{Proposition}

\newtheorem*{conj*}{Conjecture}

\newtheorem*{fact*}{Fact}

\makeatletter
\@namedef{subjclassname@2020}{%
	\textup{2020} Mathematics Subject Classification}
\makeatother

\hypersetup{
  pdftitle={The χ-Ramsey problem for triangle-free graphs},
  pdfauthor={E. Davies and F. Illingworth}
}

\begin{document}
\title{The \texorpdfstring{$\chi$}{χ}-Ramsey problem for triangle-free graphs}
\date{\today}

\author[E.\ Davies]{Ewan Davies}
\address{Department of Computer Science, University of Colorado Boulder, USA}
\email{research@ewandavies.org}
	
\author[F.\ Illingworth]{Freddie Illingworth}
\address{DPMMS, University of Cambridge, UK}
\curraddr{Mathematical Institute, University of Oxford, UK}
\email{illingworth@maths.ox.ac.uk}
	

\begin{abstract}
	In 1967, Erd\H{o}s asked for the greatest chromatic number, $f(n)$, amongst all $n$-vertex, triangle-free graphs.
	An observation of Erd\H{o}s and Hajnal together with Shearer's classical upper bound for the off-diagonal Ramsey number $R(3, t)$ shows that $f(n)$ is at most $(2 \sqrt{2} + o(1)) \sqrt{n/\log n}$.
	
	We improve this bound by a factor $\sqrt{2}$, as well as obtaining an analogous bound on the list chromatic number which is tight up to a constant factor.
	A bound in terms of the number of edges that is similarly tight follows, and these results confirm a conjecture of Cames van Batenburg, de Joannis de Verclos, Kang, and Pirot.
\end{abstract}

\maketitle

\section{Introduction}\label{sec:Intro}

The classical Ramsey question for triangle-free graphs asks for the value of $R(3, t)$\footnote{$R(3,t)$ is defined to be the least $n$ such that every $n$-vertex graph contains either a triangle or an independent set of size $t$.} or, equivalently, the smallest independence number amongst $n$-vertex triangle-free graphs.
Every graph $G$ satisfies $\alpha(G) \cdot \chi(G) \geqslant \lvert G \rvert$, which suggests a natural `$\chi$-Ramsey' question first asked by Erd\H{o}s~\cite{Erd67} in 1967. He asked for the greatest chromatic number of a triangle-free graph in terms of its number of vertices, $n$, or number of edges, $m$ (see also Appendix~B in the first edition of~\cite{ASE92}). Denoting the maximum for the vertex problem by $f(n)$, an early suggestion of the correct growth rate of $f(n)$ was given by Ajtai, Koml\'{o}s, and Szemer\'{e}di's~\cite{AKS80} upper bound $R(3, t) = \bigO(t^{2}/\log t)$, which shows that every $n$-vertex triangle-free graph has independence number at least $\Omega\bigl(\sqrt{n \log n}\bigr)$ and so \emph{Hall ratio}
\begin{equation*}
	\rho(G) = \max_{\varnothing \neq H \subseteq G} \frac{\ssize{H}}{\alpha(H)} = \bigO\bigl(\sqrt{n / \log n}\bigr).
\end{equation*}
The Hall ratio is a natural lower bound for the chromatic number. 
Erd\H{o}s and Hajnal~\cite{EH85} (see~\cite[pp.~124--5]{JT94} for details) noted that iteratively pulling out the large independent sets guaranteed by Ajtai, Koml\'{o}s, and Szemer\'{e}di's result and giving each one a different colour matches this bound. That is, $f(n) = \bigO(\sqrt{n/\log n})$. 
A well-known result of Shearer~\cite{She83} gives the best-known constant in the theorem of Ajtai, Koml\'{o}s, and Szemer\'{e}di, and this can be used to sharpen the bound to
\begin{equation*}
	f(n)\le (2\sqrt 2 + o(1))\sqrt{n/\log n}.
\end{equation*}
The classical Ramsey and $\chi$-Ramsey problems for triangle-free graphs were considered by Kim~\cite{Kim95a}, who constructed $n$-vertex, triangle-free graphs with independence number $\bigO(\sqrt{n \log n})$. 
Specifically for the $\chi$-Ramsey problem in triangle-free graphs, Kim proved that $f(n)\ge (1/9 - o(1)) \sqrt{n/\log n}$, and brought attention to the upper bound that follows from Shearer's result.
Improving Kim's lower bound, Fiz Pontiveros, Griffiths, and Morris~\cite{FPMS20} and Bohman and Keevash~\cite{BK21} followed the triangle-free process to its asymptotic end, showing that there are $n$-vertex triangle-free graphs with Hall ratio (and hence chromatic number) at least
$(1/\sqrt{2} - o(1)) \sqrt{n/ \log n}$, so there is a factor of four between these upper and lower bounds for $f(n)$.

A bound in terms of the number of edges was given by Poljak and Tuza~\cite{PT94} (see also~\cite{GT00,Nil00}), who showed that every $m$-edge triangle-free graph $G$ has chromatic number $\chi(G) = \bigO(m^{1/3}/(\log m)^{2/3})$. This is tight up to a constant factor due to Kim's construction or the triangle-free process. 
An improved constant follows from a more refined method due to Gimbel and Thomassen~\cite{GT00}, together with a theorem of Molloy that we discuss below. 
These improvements were observed in~\cite[Prop.~4.6]{CdKP20} which states the bound $\chi(G) \le (2^{1/3}3^{5/3}+o(1))m^{1/3}/(\log m)^{2/3}$.
Gimbel and Thomassen also showed that any triangle-free graph $G$ that can be embedded on an orientable or non-orientable surface of genus $g$ has $\chi(G) = \bigO(g^{1/3}/(\log g)^{2/3})$, and that this is tight up to a constant factor.

Cames van Batenburg, de Joannis de Verclos, Kang, and Pirot~\cite{CdKP20} recently highlighted the problems of tightening the asymptotic constants above, and of giving bounds for more refined graph colouring parameters including the fractional and list chromatic numbers.

\subsection{Definitions}\label{subsec:defs}

Our terminology is standard, but we briefly introduce the required notions of graph colouring here.
Recall that a \emph{$k$-colouring} of a graph is a partition of its vertex set into $k$ independent sets (colour classes). The \emph{chromatic number} of a graph $G$, denoted $\chi(G)$, is the least $k$ for which $G$ has a $k$-colouring.

A \emph{fractional colouring of weight at most $k$} is a probability distribution on the independent sets of a graph such that every vertex has probability at least $1/k$ of being in the random independent set. The \emph{fractional chromatic number}, $\chi_{f}(G)$, of a graph $G$ is the least $k$ for which $G$ has a fractional colouring of weight at most $k$. Equivalently, $\chi_{f}(G)$ is the solution to a fractional relaxation of the natural integer program that gives $\chi(G)$.

Given a graph $G$, a \emph{list assignment $L$} is an assignment of a list $L(v) \subseteq \mathbb{N}$ to each vertex $v$ of $G$. An \emph{$L$-colouring of $G$} is a colouring $c \colon V(G) \to \mathbb{N}$ that is consistent with $L$ (each vertex $v$ has $c(v) \in L(v)$) and proper ($c(u) \neq c(v)$ whenever $uv$ is an edge). The \emph{list chromatic number}, $\chi_{\ell}(G)$, of a graph $G$ is the least $k$ such that $G$ is $L$-colourable for any list assignment $L$ whose lists have size at least $k$.

It is straightforward to show that every $G$ has $\rho(G) \le \chi_{f}(G) \le \chi(G) \le \chi_{\ell}(G)$ by the definition of $\rho$, considering the uniform probability distribution over the colour classes in a $\chi(G)$-colouring of $G$, and by considering list assignments where every vertex is given the same list.

\subsection{Results}\label{subsec:results}

\Cref{table:upperbounds} summarises the best known upper bounds in terms of the number of vertices.
Given the difficulty of improving Shearer's Ramsey number bound, one should compare upper bounds for $\chi_{f}$, $\chi$, $\chi_{\ell}$ to his bound for $\rho$. 

\begin{table}
	\centering
	\caption{Upper bounds for $n$-vertex triangle-free graphs $G$}\label{table:upperbounds}
	\begin{tabular}{rlll}
		\toprule
		\multicolumn{1}{c}{Parameter} & \multicolumn{2}{c}{Previous bound} & \multicolumn{1}{c}{This work} \\
		\midrule
		$\rho(G)$ & $(\sqrt 2 + o(1)) \sqrt{n/\log n}$ & \cite{She83} & \\
		$\chi_{f}(G)$ & $(2 + o(1))\sqrt{n / \log n}$ & \cite{CdKP20} & \\
		$\chi(G)$ & $(2\sqrt 2 + o(1)) \sqrt{n / \log n}$ & \cite{She83,Kim95a} & $(2 + o(1))\sqrt{n / \log n}$\\
		$\chi_{\ell}(G)$ & $(2\sqrt 2 + o(1)) \sqrt{n}$ & \cite{CdKP20} & $(4\sqrt 2 + o(1))\sqrt{n / \log n}$\\
		\bottomrule
	\end{tabular}
\end{table}

Our first result improves the upper bound for chromatic number by a factor of $\sqrt{2}$, thus matching the bound for the fractional chromatic number established in~\cite{CdKP20}. 
We apply this new bound and a tactic of Gimbel and Thomassen~\cite{GT00} to improve the previous best bound in terms of the number of edges~\cite[Prop.~4.6]{CdKP20}.

\begin{theorem}\label{thm:chi}
	As $n \rightarrow \infty$, any triangle-free graph on $n$ vertices has chromatic number at most $(2 + o(1)) \sqrt{n / \log n}$.
	
	As $m \rightarrow \infty$, any triangle-free graph with at most $m$ edges has chromatic number at most $(3^{5/3} + o(1)) m^{1/3}/(\log m)^{2/3}$.
\end{theorem}

A particularly noticeable feature of the previous bounds in \Cref{table:upperbounds} is the upper bound for the list chromatic number, which does not have the same growth rate as the lower bound provided by the triangle-free process. We give a short argument to rectify this. 
As with the case of chromatic number, we use this new result to obtain a bound in terms of the number of edges that is tight up to a constant factor.

\begin{theorem}\label{thm:listchi}
	As $n \rightarrow \infty$, any triangle-free graph on $n$ vertices has list chromatic number at most $(4\sqrt 2 + o(1)) \sqrt{n / \log n}$.
	
	As $m \rightarrow \infty$, any triangle-free graph with at most $m$ edges has list chromatic number at most $(12\cdot 3^{2/3} + o(1)) m^{1/3}/(\log m)^{2/3}$.
\end{theorem}
\noindent
This result confirms Conjecture~6.1 from~\cite{CdKP20}, though we do not believe the given constants are tight and in particular they are worse than what we prove for the usual chromatic number.

Turning to a bound in terms of genus as studied by Gimbel and Thomassen \cite{GT00}, using \Cref{thm:listchi} we obtain the correct growth rate for the list chromatic number as well as sharpening the constant for chromatic number with an application of \Cref{thm:chi}. 
The key observation underlying the proof is that the genus and number of edges are at most a constant factor apart in the critical range.

\begin{corollary}\label{cor:genus}
	As $g\to\infty$, any triangle-free graph of genus at most $g$ has chromatic number at most $(3\cdot 6^{2/3} + o(1)) g^{1/3}/(\log g)^{2/3}$.

	As $g\to\infty$, any triangle-free graph of genus at most $g$ has list chromatic number at most $(12\cdot 6^{2/3} + o(1)) g^{1/3}/(\log g)^{2/3}$.

	The same bounds hold for graphs that can be embedded on a closed non-orientable surface of genus at most $2g$.
\end{corollary}

Due to our \Cref{thm:chi}, there is now an asymptotic factor $\sqrt{2}$ between the best known upper bounds for chromatic number and Hall ratio. Cames van Batenburg et al.~\cite{CdKP20} conjectured that it is possible to remove this $\sqrt{2}$ for the fractional chromatic number, and suggested this may be possible for the chromatic number too. 
Our results establish bounds on the list chromatic number of triangle-free graphs that are tight up to the constant factor, but one might ask whether bounds of the same order hold for an even more general notion of graph colouring known as \emph{correspondence colouring} or \emph{DP-colouring}.
In \Cref{sec:further} we discuss these constant factors, related work, some nice conjectures, and the correspondence colouring version of the $\chi$-Ramsey question.

\section{Proof Ideas and Tools}\label{sec:Sketch}

We first outline the proof of the bound in terms of the number of vertices in \Cref{thm:chi}. The key idea underpinning our improvement is induction, splitting into cases depending on the maximum degree of the graph. If some vertex of the graph has large degree, then we give its neighbourhood (which is an independent set by triangle-freeness) one colour and colour the remainder of the graph by induction. Otherwise, the graph has small maximum degree and so we may apply a result bounding the chromatic number of a triangle-free graph in terms of its maximum degree. We will use a recent ground-breaking result of Molloy~\cite{Mol19}, which is a chromatic strengthening of Shearer's lower bound for the independence number.

\begin{theorem}[Molloy]\label{thm:Molloy}
	As $\Delta \to \infty$, any triangle-free graph of maximum degree $\Delta$ has \emph{(}list\emph{)} chromatic number at most $(1 + o(1)) \Delta/\log \Delta$.
\end{theorem}

We are now in a position to sketch the proof of \Cref{thm:chi}. We will ignore all $o(1)$ terms. Let $G$ be an $n$-vertex triangle-free graph---we are trying to prove that $\chi(G) \le 2 \sqrt{n/ \log n}$ and will assume the result holds for all smaller $n$. Firstly, if $G$ has maximum degree at most $d(n) := \sqrt{n \log n}$, then Molloy's theorem immediately gives the result. Otherwise, some vertex $v$ of $G$ has degree greater than $d(n)$. Let $G'$ be $G$ with all neighbours of $v$ deleted. As $G$ is triangle-free, the neighbourhood of $v$ is an independent set, so $\chi(G) \le \chi(G') + 1$. Also, $n' = \ssize{G'} < n - d(n)$ and induction gives
\begin{equation*}
	\chi(G') \le 2 \sqrt{n'/\log n'}.
\end{equation*}
For this sketch, consider $\log n'$ and $\log n$ as identical and so
\begin{equation*}
	\chi(G') \le 2 \sqrt{\frac{n - d(n)}{\log n}} \le 2 \sqrt{\frac{n}{\log n}} - 1,
\end{equation*}
where the final inequality follows by squaring both sides and cancelling terms. As $\chi(G) \leqslant \chi(G') + 1$, we are done. The full proof has no new ideas, we merely have to overcome the technical challenge of the $o(1)$ terms. We do this in \Cref{sec:Proofs}, as well deriving the bound in terms of the number of edges.

Consider trying the same proof strategy for the list chromatic number, in pursuit of \Cref{thm:listchi}. If there is a vertex of large degree, then we cannot necessarily colour its neighbourhood with one colour as there may be no colour appearing on the lists of all its neighbours. In place of degree we use the notion of \emph{colour-degree}.

\begin{defn}
	Let $G$ be a graph with list-assignment $L$. For a vertex $v$ and a colour $c \in L(v)$, the \emph{colour-degree of $c$ at $v$} is
	\begin{equation*}
		\deg_{L}(v, c) = \ssize{\{u \colon uv \in E(G),\ c \in L(u)\}}.
	\end{equation*}
\end{defn}

Suppose that we have a graph $G$ and list assignment $L$ which assigns lists of size $k$ to the vertices of $G$.
If some colour-degree, say $\deg_{L}(v, c)$, is large, then we colour the neighbours of $v$ whose lists contain $c$ with colour $c$, remove $c$ from all other lists and delete the coloured vertices. What remains is a graph $G'$ of order $n - \deg_{L}(v, c)$, and crucially the graph $G$ admits an $L$-colouring provided $\chi_{\ell}(G') \le k-1$. 
Indeed, set $L'(u) = L(u)\setminus\{c\}$ for every vertex $u$ and note that if $\chi_{\ell}(G') \le k - 1$, then $G'$ is $L'$-colourable and the vertices in $V(G) \setminus V(G')$ can all be coloured with $c$.
We bound $\chi_{\ell}(G')$ from above by induction.
In order to carry out the same argument as above, we need an analogue of Molloy's theorem with $\Delta$ replaced by the maximum colour-degree, and such results (with larger leading constants) have been proved by Amini and Reed~\cite{AR08} and Alon and Assadi~\cite{AA20}.

\begin{theorem}[Alon--Assadi]\label{thm:AA}
	The following holds for all sufficiently large $d$. Let $G$ be a triangle-free graph with lists $L(v)$ for every vertex $v$. If, for every vertex $v$ and colour $c \in L(v)$,
	\begin{align*}
		& \ssize{L(v)} \ge 8d/ \log d, \\
		& \deg_{L}(v, c) \le d,
	\end{align*}
	then $G$ admits an $L$-colouring.
\end{theorem}

If the factor of $8$ could be replaced with $1 + o(1)$, then our bound for the list chromatic number would match those for the fractional and usual chromatic number, that is, be a factor $\sqrt{2}$ away from the bound for the Hall ratio. 
We discuss this further in \Cref{sec:further}.

The bound in terms of the number of edges given in \Cref{thm:listchi} follows from the bound in terms of the number of vertices, via a simple argument that resembles its counterpart for the chromatic number and uses a list-partitioning idea present in~\cite{CdKP20}. 
We give this argument in the next section.

\section{The Proofs}\label{sec:Proofs}

In this section we give the proofs of \Cref{thm:chi,thm:listchi}. We stress that for the bounds in terms of $n$ the main ideas appeared in \Cref{sec:Sketch}, and what remains is a technical exercise.

\begin{proof}[Proof of \Cref{thm:chi}]
	Let $f(x) = (2 + A(x)) \sqrt{x/ \log x}$ where $A = o(1)$ is smooth, non-negative and non-increasing (specified more precisely later). For the first stated bound it suffices to show that any $n$-vertex triangle-free graph $G$ has chromatic number at most $f(n)$. We will induct upon $n$, and we may choose $A$ so that the theorem holds for all $n \le 20$. Assume from now on that $n \ge 20$.
	
	First suppose that every vertex of $G$ has degree at most $d(n) = \sqrt{n \log n}$. Then, by \Cref{thm:Molloy} there is an $\eps(x) = o(1)$ such that
	\begin{align*}
		\chi(G) & \le (1 + \eps(n)) \frac{d(n)}{\log d(n)} \le (1 + \eps(n)) \frac{d(n)}{\log(n^{1/2})} \\
		& = 2(1 + \eps(n)) \sqrt{\frac{n}{\log n}}.
	\end{align*}
	Thus, we are done in this case provided 
	\begin{equation}\label{eq:eps}
		A(x) \geqslant 2 \varepsilon(x).
	\end{equation}
	In the second case there is some vertex $v$ with degree greater than $d(n)$. Let $G'$ be the graph obtained from $G$ by deleting all the neighbours of $v$. Then $G'$ has fewer than $n - d(n)$ vertices and
	\begin{equation*}
		\chi(G) \le \chi(G') + 1 \le f(n - d(n)) + 1,
	\end{equation*}
	where the second inequality follows by induction. Hence, to complete the proof we need 
	\begin{equation}\label{eq:f}
		f(n) - f(n - d(n)) \ge 1,
	\end{equation}
	for all $n \ge 20$. It remains to check that it is possible to choose an $A$ such that equations~\Cref{eq:eps,eq:f} hold. We will assume that $A$ decays sufficiently slowly so that equation~\Cref{eq:eps} holds.
	
	The function $\sqrt{x/\log x}$ is concave for $x \ge 6$. If we choose the non-negative function $A$ decaying sufficiently slowly, then $f$ ought to be concave too. Indeed, if we choose $A$ so that $\abs{A''(x)} \le 1/(10 x^{2})$, then a quick calculation\footnote{With $g(x) = \sqrt{x/\log x}$ we have $f'(x) = (2 + A(x))g'(x) + A'(x)g(x)$ and $g'(x)=\frac{\log x - 1}{2\sqrt{x (\log x)^{3}}}$.} shows that for all $x \ge 10$, $f''(x) < 0$ (note that $A' \le 0 \le A$).
	By concavity, for all $n \ge 20$,
	\begin{equation*}
		f(n) - f(n - d(n)) \ge f'(n) d(n) = \bigl(1 + \tfrac{A(n)}{2}\bigr)\bigl(1 - \tfrac{1}{\log n}\bigr) + nA'(n).
	\end{equation*}
	Choosing $A$ so that $\abs{A'(x)} \le 1/(x \log x)$ and $A(x) \ge 8/\log x$ gives equation~\Cref{eq:f}. One should worry that the conditions we have placed on the derivatives of $A$ might preclude it from tending to zero. Happily, integrating these shows that this is not the case.

	Now for the bound in terms of the number of edges. We apply the method of Gimbel and Thomassen~\cite{GT00}. For $d = ((m \log m)/3)^{1/3}$, let $V_{1} \subseteq V(G)$ be those vertices of degree at most $d$ and let $V_{2} = V(G) \setminus V_{1}$.
	
	The subgraph of $G$ induced by $V_{1}$ can be properly coloured with at most $(1 + o(1))d/\log d$ colours by \Cref{thm:Molloy}. On the other hand, $2m \geqslant \sum_{v \in V_{2}} \deg(v) \geqslant \ssize{V_{2}}d$ and so $\ssize{V_{2}} \leqslant 2m/d = 2 \cdot 3^{1/3} \cdot m^{2/3}/(\log m)^{1/3}$. 
	Applying the first part of this theorem we obtain
	\begin{displaymath}
		\chi(G) \le (1+o(1))\frac{d}{\log d} + (2+o(1))\sqrt{\frac{\ssize{V_{2}}}{\log \ssize{V_{2}}}} \le (3^{3/5} + o(1))\frac{m^{1/3}}{(\log m)^{2/3}}.
	\end{displaymath}
\end{proof}

Before we prove \Cref{thm:listchi}, we recall the Chernoff bound.

\begin{theorem}[{Chernoff Bound, \cite[Cor.~2.3]{JLR00}}]
	Fix $p \in (0, 1)$ and let the random variable $X \sim \mathrm{Bin}(k,p)$ be binomially distributed. 
	Then for $\eps\in(0,1)$,
	\begin{equation*}
		\Pr(X \le (1-\eps)kp) \le 2e^{-\eps^2 kp/3}.
	\end{equation*}
\end{theorem}

\begin{proof}[Proof of \Cref{thm:listchi}]
	Let $g(x) = (4 \sqrt{2} + B(x)) \sqrt{x/\log x}$ where $B = o(1)$ is smooth, non-negative and non-increasing.  We may choose $B$ so that \Cref{thm:listchi} holds for all small $n$. Assume from now on that $n \ge n_{0}$ for some fixed $n_{0}$.
	
	Associate with each vertex $v$ a list $L(v)$ of colours such that $\ssize{L(v)} \ge g(n)$. It suffices to show that there is a proper colouring of $G$ from these lists. First suppose all colour-degrees are at most $d(n) = \sqrt{2}/4 \cdot \sqrt{n \log n}$. Now
	\begin{equation*}
		\frac{8 d(n)}{\log d(n)} \le \frac{8d(n)}{\log(n^{1/2})} \le g(n),
	\end{equation*}
	where the first inequality holds provided $n_{0} \ge e^{8}$. Provided $n_{0}$ is large enough, \Cref{thm:AA} guarantees that there is an $L$-colouring of $G$.
	
	Otherwise, there is some vertex $v$ and colour $c \in L(v)$ with $\deg_{L}(v, c) > d(n)$. Let $G'$ be the graph obtained from $G$ by deleting all the neighbours of $v$ with colour $c$ on their list. Then $G'$ has fewer than $n - d(n)$ vertices, and $G$ admits an $L$-colouring provided $\chi_{\ell}(G') \le g(n) - 1$. 
	By induction, we have $\chi_\ell(G') \le g(n-d(n))$ and hence to show $\chi_\ell(G)\le g(n)$ it suffices to prove that $g(n) - g(n - d(n)) \ge 1$.
	We finish as in the proof of \Cref{thm:chi}: choosing the non-negative function $B$ so that $\abs{B''(x)} \le 1/(10x^{2})$ guarantees that $g$ is concave for $x \ge 10$. Thus, for $n \ge 20$
	\begin{align*}
		g(n) - g(n - d(n)) &\ge g'(n) d(n) \\
		& = \bigl(1 + B(n) \cdot \tfrac{\sqrt{2}}{8}\bigr) \bigl(1 - \tfrac{1}{\log n}\bigr) + nB'(n) \cdot \tfrac{\sqrt{2}}{4}.
	\end{align*}
	Choosing $B$ so that $\abs{B'(x)} \le 1/(x \log x)$ and $B(x) \ge 32/\log x$ gives $\chi_{\ell}(G) \le g(n)$, as required.
	
	For the second part of \Cref{thm:listchi}, we partition $V(G)$ according to the largest colour-degree at each vertex. Let $\eps>0$ be arbitrary and without loss of generality suppose that $\eps<1/2$ and $m$ is sufficiently large.
	Let $G$ be a triangle-free graph on $n$ vertices with $m$ edges, and consider a list assignment $L$ giving each vertex a list of size at least
	\begin{equation*}
		k = (1 + \eps)^2 \cdot 12 \cdot 3^{2/3} \frac{m^{1/3}}{(\log m)^{2/3}}.
	\end{equation*}
	For $d = ((m\log m)/24)^{1/3}$, consider the partition of $V(G)$ given by 
	\begin{align*}
		V_1 & = \{v \in V(G) \colon \max_{c \in L(v)} \deg_L(v,c) \le d\}, \\
		V_2 & = V(G) \setminus V_1.
	\end{align*}
	Every vertex in $V_2$ has some colour-degree at least $d$ and so has degree at least $d$. Hence, degree-counting gives $\ssize{V_{2}} \le 2m/d$.
	We partition the colours in $\bigcup_{v \in V(G)} L(v)$ into two parts $L_{1}$ and $L_{2}$ and let $L_{i}(v) = L(v) \cap L_{i}$. We insist on using the colours in $L_{1}$ to colour $G[V_{1}]$ and those in $L_{2}$ to colour $G[V_{2}]$. Assume for now that every vertex $v$ satisfies:
	\begin{equation}\label{eq:minL}
		\begin{aligned}
		\ssize{L_1(v)} & \ge (4 \cdot 3^{2/3} + \eps)\frac{m^{1/3}}{(\log m)^{2/3}}, \\
		\ssize{L_{2}(v)} & \ge (8 \cdot 3^{2/3} + \eps) \frac{m^{1/3}}{(\log m)^{2/3}}.
		\end{aligned}
	\end{equation}
	For every vertex $v \in V_1$ and $c \in L_1(v)$: $\deg_{L_1}(v, c) \le \deg_L(v, c) \le d$ and
	\begin{equation*}
		\frac{8d}{\log d} \le (4\cdot 3^{2/3}+o(1))\frac{m^{1/3}}{(\log m)^{2/3}} \le \ssize{L_1(v)}.
	\end{equation*}
	By \Cref{thm:AA}, $G[V_1]$ is $L_{1}$-colourable (provided $m$ is large enough). Next, by the first part of this theorem,
	\begin{equation*}
		\chi_\ell(G[V_2])\le (4+o(1))\sqrt{\frac{2\ssize{V_2}}{\log\ssize{V_2}}} \le (8\cdot 3^{2/3} + o(1))\frac{m^{1/3}}{(\log m)^{2/3}},
	\end{equation*}
	so $G[V_2]$ is $L_2$-colourable. As $L_1$ and $L_2$ are disjoint sets of colours, these colourings can be combined to give an $L$-colouring of $G$.
	
	We finally check that there is a partition $L_1 \cup L_2$ of the colours for which~\Cref{eq:minL} holds. Each colour $c \in \bigcup_{v \in V(G)} L(v)$ is independently placed in $L_{1}$ with probability $1/3$ and is otherwise placed in $L_{2}$. By the Chernoff bound, each vertex $v\in V(G)$ has
	\begin{align*}
		\Pr\biggl(\ssize{L_1(v)} \le (4\cdot 3^{2/3} + \eps)\frac{m^{1/3}}{(\log m)^{2/3}}\biggr) & \le \Pr(\ssize{L_1(v)} \le (1-\eps)k/3), \\
		& \le 2e^{-\eps^2k/9}.
	\end{align*}
	Similarly, each vertex $v$ has
	\begin{equation*}
		\Pr\biggl(\ssize{L_2(v)} \le (8\cdot 3^{2/3} + \eps)\frac{m^{1/3}}{(\log m)^{2/3}}\biggr) \le 2e^{-2\eps^2k/9}.
	\end{equation*}
	
	We may assume that $G$ has no vertices of degree $0$ or $1$, as any such vertices can be removed before colouring $G$ and compatible colours found when they are re-added. 
	Then $G$ has at most $m$ vertices, and so $k > n^{1/3}/(\log n)^{2/3}$ and $2e^{-\eps^2k/9} = o(1/n)$. A union bound over the $n$ vertices of $G$ now gives that there is a partition of the lists of colours that satisfies~\Cref{eq:minL}.
\end{proof}

We conclude this section with the short proof of \Cref{cor:genus}, which is simply a slightly more precise version of the argument given in~\cite{GT00}. 

\begin{proof}[Proof of \Cref{cor:genus}]
	Let $G$ be a graph on $n$ vertices, with $m$ edges and genus at most $g$. We may assume that $G$ is connected.
	By removing vertices of degree at most $d = g^{1/3}(\log g)^{-2/3}$, which can then be coloured as they are re-added, we may assume that $G$ has minimum degree at least $d$. 
	Then the number $m$ of edges of $G$ satisfies $2m \ge dn$, and hence as $g\to\infty$ we have $n = o(m)$.
	
	Now consider an embedding of $G$ on an orientable surface of genus $g$ such that the drawing of $G$ has $r$ regions. Since $G$ is triangle-free, each region is surrounded by at least four edges. Also, each edge bounds at most two regions and so $r\le m/2$.
	Euler's formula gives $n - m/2 \ge n - m + r \ge 2-2g$, and hence $m \le (4+o(1))g$. 
	The corollary now follows by applications of \Cref{thm:chi,thm:listchi}. 
	For non-orientable genus the argument is the same, but we must use Euler's formula in the form $n-m+r \ge 2 - k$ where $k$ is the non-orientable genus.
\end{proof}

\section{Related research and open problems}\label{sec:further}

\subsection{Fractional colouring}
Cames van Batenburg et al.~\cite[Conjs.~4.3, 4.4]{CdKP20} conjectured that upper bounds for the fractional chromatic number in terms of $n$ and $m$ should match Shearer's bound on the Hall ratio in triangle-free graphs.
Recall that a fractional colouring of weight at most $k$ is a probability distribution on independent sets such that for every vertex $v$ we have $\Pr(v\in I)\ge 1/k$ and so it can be particularly useful to study distributions with a `local' lower bound on $\Pr(v\in I)$ that depends on parameters such as $\deg(v)$.

In a triangle-free graph on $n$ vertices, taking a uniform random neighbourhood gives $\Pr(v\in I) = \deg(v)/n$ and combining this distribution with one derived from a suitable local version of Molloy's theorem should perform well. 
A local fractional colouring result from~\cite{DJKP18} associates to each vertex $v$ of a triangle-free graph a subset $w(v)$ of the positive reals of measure $1$ such that $w(u)$ and $w(v)$ are disjoint for edges $uv$, and $w(v) \subset [0,m_v)$ for some $m_v=(1+o(1))\deg(v)/\log \deg(v)$. Choosing a positive real number at random with a non-increasing density function such as 
\[ p(r) = \max\left\{0,\sqrt{\frac{n}{\log n}}-\frac{r}{2}\frac{n}{\log n}\right\} \] 
gives a random independent set $\{v \colon r\in w(v)\}$ which is more likely to contain vertices of lower degree. 
Combining this with the distribution obtained by taking the neighbourhood of a uniformly random vertex, one recovers the fractional case of \Cref{thm:chi} (originally proved in~\cite{CdKP20} by a slightly more involved argument). 

Kelly and Postle~\cite{KP18} conjectured an improved (and more natural) version of a distribution which favours independent sets containing low degree vertices, and specifically noted that combining this with the `random neighbourhood' distribution matches Shearer's bound.
Their conjecture is that every triangle-free graph admits a probability distribution on independent sets such that for every vertex $v$, $\Pr(v\in I) \ge (1-o(1))\log\deg(v)/\deg(v)$.

\subsection{List and correspondence colouring}
A generalisation of list colouring known as \emph{correspondence colouring} or \emph{DP-colouring} permits an arbitrary matching of colours from $L(u)$ to $L(v)$ to be `forbidden' at the edge $uv$. 
Formally, given a list assignment $L$ for a graph $G$, we consider the sets $\hat L(v) = \{ (v, c) : c\in L(v) \}$ and a graph $H$ on vertex set $\bigcup_{v\in V(G)}\hat L(v)$. 
To form $H$ we put a clique on each $\hat L(v)$ and for all edges $uv\in E(G)$ and colours $c\in L(u)\cap L(v)$ we connect each $(u,c)$ to $(v,c)$. This gives a \emph{cover graph} such that independent sets in $H$ of size $|V(G)|$ in $H$ are in 1-to-1 correspondence with the $L$-colourings of $G$.
Correspondence colouring arises when we relax the requirements on the edges of $H$ between each $\hat L(u)$ and $\hat L(v)$, instead allowing an arbitrary matching between $\hat L(u)$ and $\hat L(v)$ for each edge $uv$ of $G$. 
The \emph{correspondence chromatic number} of $G$, denoted $\chi_c(G)$, is then the least $k$ such that whenever we start with a list assignment with lists of size $k$ and construct a cover $H$ in this fashion with arbitrary matchings, the resulting $H$ contains an independent set of size $|V(G)|$. 
This definition and comparison to list colouring shows that we have $\chi_\ell(G)\le\chi_c(G)$ for any graph $G$.
Correspondence colouring is well-studied in our setting, for example Bernshteyn~\cite{Ber19} showed that \Cref{thm:Molloy} holds for the correspondence chromatic number. 
Going further, Cambie and Kang~\cite{CK20} conjectured a version of \Cref{thm:AA} for correspondence colouring and with $8$ replaced by $1+o(1)$, having proved this in the special case of bipartite graphs. 

Local versions of the results we rely on are natural in the settings of list and correspondence colouring, where the necessary lower bound on $|L(v)|$ is a function of $\deg(v)$. 
See~\cite{BKNP18,DJKP18,DKPS20} for bounds in terms of degrees, and~\cite[Conj.~9.3.2]{Kel19} for a conjectured bound in terms of colour-degrees for list colouring. 
It is natural to combine~\cite[Conj.~9.3.2]{Kel19} and~\cite[Conj.~4]{CK20} and propose a local colour-degree version of \Cref{thm:Molloy} in the case of correspondence colouring (though some care with a minimum list size is necessary, see~\cite{DJKP18}).

It is interesting to note that the solution to the `$\chi_c$-Ramsey' problem for triangle-free graphs follows easily (up to a constant factor) from results of Bernshteyn~\cite{Ber16,Ber19} and Kr\'a\v{l}, Pangr\'ac, and Voss~\cite{KPV05}. 
The independent works~\cite{Ber16,KPV05} contain the lower bound $\chi_c(G) \ge (1/2-o(1))d/\log d$ for arbitrary graphs $G$ of average degree $d$, so a balanced $n$-vertex complete bipartite graph has $\chi_c(K_{\lfloor n/2\rfloor,\lceil n/2\rceil})\ge (1/4-o(1))n/\log n$.
Then in~\cite{Ber19} Bernshteyn gave the strengthening of \Cref{thm:Molloy} to correspondence colouring and, since the maximum degree is at most the number of vertices, showed that when $G$ is an $n$-vertex triangle-free graph we have $\chi_c(G) \le (1+o(1)) n/\log n$. This is therefore tight up to a factor of at most $4$. This correspondence colouring problem behaves very differently from the list colouring problem: balanced complete bipartite graphs are perhaps the extremal example and the extremal value is $\Theta(n/\log n)$ rather than its square root.

\subsection{Relaxing the triangle-free condition}

Colouring graphs with sparse neighbourhoods is natural generalisation of the problems we mention here. 
A triangle-free graph has independent neighbourhoods, but this can be relaxed to neighbourhoods inducing a bounded number of edges, or bounded (fractional) chromatic number, Hall ratio or clique number\footnote{Consider a sequence of assumptions inspired by the inequalities $\omega \le \rho \le \chi_f\le \chi$ in subgraphs induced by neighbourhoods.}.
A classic problem in Ramsey theory is to upper bound the Hall ratio of $K_4$-free graphs in terms of the number of vertices, where the best-known bound is due to Li, Rousseau, and Zang~\cite{LRZ01}. 
We do not dwell on the $\chi$-Ramsey version of this problem as there is still a polynomial gap between this upper bound and the lower bound due to Bohman~\cite{Boh09} proved by analysing the $K_4$-free process, so improving the constant factor with our methods is an unedifying prospect.

We focus on the $\chi$-Ramsey question for triangle-free graphs because known upper and lower bounds on the Hall ratio differ by only a constant factor, and we now match the growth rate for the list chromatic number too.
Much less is known for more general sparse neighbourhood conditions and there are usually large gaps between known upper and lower bounds on Hall ratio. 
Notably, the case of graphs in which neighbourhoods induce a bounded number of edges was settled up to a small constant factor in~\cite{DJKP21}.
One can always transfer upper bounds on the Hall ratio to ones on chromatic number by iteratively pulling out large independent sets as colour classes (see~\cite[pp.~124--5]{JT94} or~\cite[Lem.~4.1]{CdKP20}), and this raises the question of whether our methods can improve upon such arguments.
Broadly, the answer is yes as generalisations of \Cref{thm:Molloy} for these settings with good leading asymptotic constants have been given in~\cite{AIS19a,DKPS20} which means that one can prove analogues of our \Cref{thm:chi} with the same method.
Rather than chasing constants in the upper bounds on chromatic number, we suggest that it would be interesting to focus on improving upper bounds for the Hall ratio, and to find good lower bounds. 
Both of these are deep and challenging problems.

The situation for list chromatic number is rather different, however, as there is no obvious, generic way to transfer bounds on the Hall ratio to list chromatic number without losing some factor of $(\log n)^c$ (see e.g.~\cite[Thm.~6.4]{CdKP20} for an argument that can be generalised). 
The `$\chi_\ell$-Ramsey' question for graphs in which neighbourhoods induce a bounded number of edges was in fact already asked in~\cite{DJKP21}.
Our proof of \Cref{thm:listchi} shows how one might match the growth rate of Hall ratio bounds in this setting and others, but we do not have the necessary analogues of \Cref{thm:AA} (which gives a bound in terms of colour-degree) in the more general sparse neighbourhood settings. 
The $\chi_\ell$-Ramsey problem motivates the pursuit of such results, though they are certainly worth investigating in their own rights too.

\section{Acknowledgements}

We would like to thank Ross Kang and Jean-S\'ebastien Sereni for organising the March 2021 workshop `Entropy Compression and Related Methods', where we began this work.
We thank Ross Kang in particular for sharing the problems we study here.

We are grateful to the anonymous referees for their comments, in particular for suggesting an improvement to the constant in the second bound in \Cref{thm:listchi}.

\bibliographystyle{habbrv}
\bibliography{../colouring-abbrv}

\begin{thebibliography}{10}
\expandafter\ifx\csname url\endcsname\relax
  \def\url#1{\texttt{#1}}\fi
\expandafter\ifx\csname doi\endcsname\relax
  \def\doi#1{\burlalt{\textsc{doi}:\detokenize{#1}}{https://dx.doi.org/#1}}\fi
\expandafter\ifx\csname
  urlprefix\endcsname\relax\def\urlprefix{\textsc{url:}}\fi
\expandafter\ifx\csname href\endcsname\relax
  \def\href#1#2{#2}\fi
\expandafter\ifx\csname burlalt\endcsname\relax
  \def\burlalt#1#2{\href{#2}{#1}}\fi

\bibitem{AIS19a}
D.~Achlioptas, F.~Iliopoulos, and A.~Sinclair.
\newblock Beyond the {{Lov\'asz Local Lemma}}: {{Point}} to {{Set
  Correlations}} and {{Their Algorithmic Applications}}.
\newblock In {\em 2019 {{IEEE}} 60th {{Annual Symposium}} on {{Foundations}} of
  {{Computer Science}} ({{FOCS}})}, pages 725--744, {Baltimore, MD, USA}, Nov.
  2019. {IEEE}.
\newblock \doi{10.1109/FOCS.2019.00049}.

\bibitem{AKS80}
M.~Ajtai, J.~Koml{\'o}s, and E.~Szemer{\'e}di.
\newblock A note on {R}amsey numbers.
\newblock {\em J. Combin. Theory Ser. A}, 29(3):354--360, Nov. 1980.
\newblock \doi{10.1016/0097-3165(80)90030-8}.

\bibitem{AA20}
N.~Alon and S.~Assadi.
\newblock Palette {{Sparsification Beyond}} {{$(\Delta + 1)$}} {{Vertex
  Coloring}}.
\newblock In J.~Byrka and R.~Meka, editors, {\em Approximation,
  {{Randomization}}, and {{Combinatorial Optimization}}. {{Algorithms}} and
  {{Techniques}} ({{APPROX}}/{{RANDOM}} 2020)}, volume 176 of {\em Leibniz
  {{International Proceedings}} in {{Informatics}} ({{LIPIcs}})}, pages
  6:1--6:22, {Dagstuhl, Germany}, 2020. {Schloss Dagstuhl--Leibniz-Zentrum
  f\"ur Informatik},
  \burlalt{arXiv:2006.10456}{https://arxiv.org/abs/2006.10456}.

\bibitem{ASE92}
N.~Alon, J.~H. Spencer, and P.~Erd{\H o}s.
\newblock {\em The Probabilistic Method}.
\newblock Wiley-{{Interscience}} Series in Discrete Mathematics and
  Optimization. {Wiley}, {New York}, 1992.

\bibitem{AR08}
O.~Amini and B.~Reed.
\newblock List {{Colouring Constants}} of {{Triangle Free Graphs}}.
\newblock {\em Electron. Notes Discrete Math.}, 30:135--140, Feb. 2008.
\newblock \doi{10.1016/j.endm.2008.01.024}.

\bibitem{Ber16}
A.~Bernshteyn.
\newblock The asymptotic behavior of the correspondence chromatic number.
\newblock {\em Discrete Math.}, 339(11):2680--2692, Nov. 2016.
\newblock \doi{10.1016/j.disc.2016.05.012}.

\bibitem{Ber19}
A.~Bernshteyn.
\newblock The {{Johansson}}-{{Molloy}} theorem for {{DP}}-coloring.
\newblock {\em Random Structures Algorithms}, 54(4):653--664, July 2019.
\newblock \doi{10.1002/rsa.20811}.

\bibitem{Boh09}
T.~Bohman.
\newblock The triangle-free process.
\newblock {\em Adv. Math.}, 221(5):1653--1677, Aug. 2009.
\newblock \doi{10.1016/j.aim.2009.02.018}.

\bibitem{BK21}
T.~Bohman and P.~Keevash.
\newblock Dynamic concentration of the triangle-free process.
\newblock {\em Random Structures Algorithms}, 58(2):221--293, 2021.
\newblock \doi{10.1002/rsa.20973}.

\bibitem{BKNP18}
M.~Bonamy, T.~Kelly, P.~Nelson, and L.~Postle.
\newblock Bounding {$\chi$} by a fraction of {$\Delta$} for graphs without
  large cliques.
\newblock Mar. 2018,
  \burlalt{arXiv:1803.01051}{https://arxiv.org/abs/1803.01051}.

\bibitem{CK20}
S.~Cambie and R.~J. Kang.
\newblock Independent transversals in bipartite correspondence-covers.
\newblock {\em Can. Math. Bulletin}, pages 1--13, Dec. 2021.
\newblock \doi{10.4153/S0008439521001004}.

\bibitem{CdKP20}
W.~{Cames van Batenburg}, R.~{de Joannis de Verclos}, R.~J. Kang, and F.~Pirot.
\newblock Bipartite {{Induced Density}} in {{Triangle}}-{{Free Graphs}}.
\newblock {\em Electron. J. Combin.}, 27(2.34), May 2020.
\newblock \doi{10.37236/8650}.

\bibitem{DJKP18}
E.~Davies, R.~{de Joannis de Verclos}, R.~J. Kang, and F.~Pirot.
\newblock Coloring triangle-free graphs with local list sizes.
\newblock {\em Random Structures Algorithms}, 57(3):730--744, 2020.
\newblock \doi{10.1002/rsa.20945}.

\bibitem{DJKP21}
E.~Davies, R.~{de Joannis de Verclos}, R.~J. Kang, and F.~Pirot.
\newblock Occupancy fraction, fractional colouring, and triangle fraction.
\newblock {\em J. Graph Theor.}, 97(4):557--568, 2021.
\newblock \doi{10.1002/jgt.22671}.

\bibitem{DKPS20}
E.~Davies, R.~J. Kang, F.~Pirot, and J.-S. Sereni.
\newblock Graph structure via local occupancy.
\newblock Mar. 2020,
  \burlalt{arXiv:2003.14361}{https://arxiv.org/abs/2003.14361}.

\bibitem{EH85}
P.~Erd\H{o}s and A.~Hajnal.
\newblock Chromatic number of finite and infinite graphs and hypergraphs.
\newblock {\em Discrete Math.}, 53:281--285, Mar. 1985.
\newblock \doi{10.1016/0012-365X(85)90148-7}.

\bibitem{Erd67}
P.~Erd{\H o}s.
\newblock Some remarks on chromatic graphs.
\newblock {\em Colloq. Math.}, 16:253--256, 1967.
\newblock \doi{10.4064/cm-16-1-253-256}.

\bibitem{FPMS20}
G.~Fiz~Pontiveros, S.~Griffiths, and R.~Morris.
\newblock The triangle-free process and the {R}amsey number {$R(3,k)$}.
\newblock {\em Mem. Amer. Math. Soc.}, 263(1274), 2020.
\newblock \doi{10.1090/memo/1274}.

\bibitem{GT00}
J.~Gimbel and C.~Thomassen.
\newblock Coloring triangle-free graphs with fixed size.
\newblock {\em Discrete Math.}, 219(1):275--277, May 2000.
\newblock \doi{10.1016/S0012-365X(00)00087-X}.

\bibitem{JLR00}
S.~Janson, T.~{\L}uczak, and A.~Rucinski.
\newblock {\em Random Graphs}.
\newblock Wiley-{{Interscience}} Series in Discrete Mathematics and
  Optimization. {Wiley}, {New York}, 2000.

\bibitem{JT94}
T.~R. Jensen and B.~Toft.
\newblock {\em Graph Coloring Problems}.
\newblock John Wiley \& Sons, Ltd, 1994.
\newblock \doi{10.1002/9781118032497}.

\bibitem{Kel19}
T.~Kelly.
\newblock {\em Cliques, {{Degrees}}, and {{Coloring}}: {{Expanding}} the
  {$\omega$}, {{$\Delta$}}, {$\chi$} Paradigm}.
\newblock PhD thesis, University of Waterloo, 2019.
\newblock \urlprefix\url{http://hdl.handle.net/10012/14862}.

\bibitem{KP18}
T.~Kelly and L.~Postle.
\newblock Fractional coloring with local demands.
\newblock Nov. 2018,
  \burlalt{arXiv:1811.11806}{https://arxiv.org/abs/1811.11806}.

\bibitem{Kim95a}
J.~H. Kim.
\newblock The {{Ramsey}} number {$R(3, t)$} has order of magnitude {$t^2/\log
  t$}.
\newblock {\em Random Structures Algorithms}, 7(3):173--207, Oct. 1995.
\newblock \doi{10.1002/rsa.3240070302}.

\bibitem{KPV05}
D.~Kr{\'a}{\v l}, O.~Pangr{\'a}c, and H.-J. Voss.
\newblock {A note on group colorings}.
\newblock {\em J. Graph Theor.}, 50(2):123--129, 2005.
\newblock \doi{10.1002/jgt.20098}.

\bibitem{LRZ01}
Y.~Li, C.~C. Rousseau, and W.~Zang.
\newblock Asymptotic {{Upper Bounds}} for {{Ramsey Functions}}.
\newblock {\em Graphs Combin.}, 17(1):123--128, Mar. 2001.
\newblock \doi{10.1007/s003730170060}.

\bibitem{Mol19}
M.~Molloy.
\newblock The list chromatic number of graphs with small clique number.
\newblock {\em J. Combin. Theory Ser. B}, 134:264--284, Jan. 2019.
\newblock \doi{10.1016/j.jctb.2018.06.007}.

\bibitem{Nil00}
A.~Nilli.
\newblock Triangle-free graphs with large chromatic numbers.
\newblock {\em Discrete Math.}, 211(1):261--262, Jan. 2000.
\newblock \doi{10.1016/S0012-365X(99)00109-0}.

\bibitem{PT94}
S.~Poljak and Z.~Tuza.
\newblock Bipartite {{Subgraphs}} of {{Triangle}}-{{Free Graphs}}.
\newblock {\em SIAM J. Discrete Math.}, 7(2):307--313, May 1994.
\newblock \doi{10.1137/S0895480191196824}.

\bibitem{She83}
J.~B. Shearer.
\newblock A note on the independence number of triangle-free graphs.
\newblock {\em Discrete Math.}, 46(1):83--87, 1983.
\newblock \doi{10.1016/0012-365X(83)90273-X}.

\end{thebibliography}
\end{document}